\documentclass[11pt]{amsart}
\usepackage[colorlinks=true, pdfstartview=FitV, linkcolor=blue, 
citecolor=blue]{hyperref}

\usepackage{mathrsfs,amssymb,amsmath,amsfonts}
\usepackage{bm}
\usepackage{enumerate}
\usepackage{array}
\usepackage{color}

\usepackage[letterpaper,asymmetric,left=1in,right=1in,top=1.2in,bottom=1.  25in,bindingoffset=0.0in]{geometry}

\newcommand{\NN}{\mathbb{N}}
\newcommand{\RR}{\mathbb{R}}
\newcommand{\CC}{\mathbb{C}}
\newcommand{\ZZ}{\mathbb{Z}}
\newcommand{\QQ}{\mathbb{Q}}
\newcommand{\rar}{\rightarrow}
\newcommand{\mc}{\mathcal}

\newcommand{\E}{\mathcal{E}}
\renewcommand{\S}{\mathcal{S}}
\newcommand{\R}{\mathcal{R}}

\newtheorem*{theorem*}{Theorem}
\newtheorem{theorem}{Theorem}
\newtheorem{lemma}[theorem]{Lemma}

\newtheorem{proposition}[theorem]{Proposition}

\newtheorem{remark}[theorem]{Remark}
\newtheorem{question}{Question}

\newtheorem*{THMA}{Theorem A}
\newtheorem*{THMB}{Theorem B}
\newtheorem*{THMAPRIME}{Theorem A'}
\newtheorem*{THMBPRIME}{Theorem B'}
\newtheorem*{THMC}{Theorem C}
\newtheorem*{THM_BD}{Boshernitzan-Dyson Theorem}

\begin{document}

\title[Level
spacing statistics for the multi-dimensional quantum harmonic oscillator]{Level
spacing statistics for the multi-dimensional quantum harmonic oscillator:
algebraic case}

\author{Alan Haynes, Roland Roeder}
\date{\today}

\thanks{MSC 2020: 81Q50,  11J13, 11J17}

\keywords{Quantum Harmonic Oscillator, Energy Level Statistics, Steinhaus problem, Three Gap Theorem}

\begin{abstract}
We study the statistical properties of the spacings between neighboring energy
levels for the multi-dimensional quantum harmonic oscillator that occur in a
window $[E,E+\Delta E)$ of fixed width $\Delta E$ as $E$ tends to infinity.  This regime provides a notable exception to the Berry-Tabor Conjecture from Quantum Chaos and,
for that reason, it was studied extensively by Berry and Tabor in their seminal
paper from 1977.  We focus entirely on the case that the (ratios of)
frequencies $\omega_1,\omega_2,\ldots,\omega_d$ together with $1$ form a basis
for an algebraic number field $\Phi$ of degree $d+1$, allowing us to use tools
from algebraic number theory.  This special case was studied by Dyson, 
Bleher, Bleher-Homma-Ji-Roeder-Shen, and others.  Under a suitable rescaling,
we prove that the distribution of spacings behaves asymptotically
quasiperiodically in $\log E$.  We also prove that the distribution of ratios
of neighboring spacings behaves asymptotically quasiperiodically in $\log E$.
The same holds for the distribution of finite words in the finite alphabet of
rescaled spacings.

Mathematically, our work is a  higher dimensional version of the Steinhaus
Conjecture (Three Gap Theorem) involving the fractional parts of a linear form in
more than one variable, and it is of independent interest from this
perspective.
\end{abstract}

\maketitle

\section{Introduction}
We study the statistical properties of spacings (also called gaps and nearest neighbor distances) between neighboring
energy levels for the multi-dimensional quantum harmonic oscillator in the
limit as the energy tends to infinity.  This is a fundamental problem from
physics whose motivations date back to the origins of the study of Quantum Chaos.

Our main techniques are from algebraic number theory and to make these
techniques accessible we will focus entirely on the case that the (ratios of)
frequencies $\omega_1,\omega_2,\ldots,\omega_d$ together with $1$ form a basis
for an algebraic number field $\Phi$ of degree $d+1$.  (Here the dimension of the harmonic oscillator is $d+1$.)
This is a rather special
case, however it is also very important because of the emphasis placed
on it by previous works of Berry and Tabor \cite{BT}, Dyson \cite{Dyson}, Bleher
\cite{B1}, and others.

Mathematically, the problem reduces to studying the fractional parts of a
linear form in more than one variable, which is a topic of independent interest
in number theory.  Study of this topic provides multi-dimensional analogs and extensions of the famous
Three Gap Theorem (also called the Steinhaus Conjecture) of S\'{o}s
\cite{Sos1957},  Sur\'{a}nyi \cite{Sura1958}, and \'{S}wierczkowski
\cite{Swie1959}.  

In order to reach readers from both of these perspectives we will now describe the
motivations, background, and also state our results in both contexts.  

\subsection{Physical Perspective}\label{SUBSEC:PHYSICAL}
Given a quantum-mechanical system (or equivalently a Hamiltonian~$H$),
one of the central problems from Quantum Chaos is to study the distribution of spacings between neighboring energy levels
in the limit as the energy $E$ tends to infinity.  One typically considers all energy levels 
in a window
\begin{align*}
E \leq e_1 < e_2 < \ldots < e_\ell < E + \Delta E,
\end{align*}
where the width $\Delta E$ is fixed and one takes the limit as $E \rightarrow
\infty$.  In this way, the ordered energy levels $e_1(E) < e_2(E) < \ldots <
e_{\ell(E)}(E)$ all depend on the energy $E > 0$ as does the number $\ell
\equiv \ell(E)$ of them.   One is interested in the statistics of the distances between neighboring energy levels
\begin{align*}
\delta_i(E) := e_{i+1}(E) - e_{i}(E)\qquad \mbox{for $1 \leq i \leq \ell(E)-1$},
\end{align*}
and their fluctuations in the limit as $E \rightarrow \infty$.  
As $E$ increases, the size of these spacings decrease.  However, 
one can multiply by a factor $\lambda(E)$ in order to produce normalized spacings
\begin{align*}
\overline{\delta}_i(E) &:= \lambda(E) \delta_i(E) = \lambda(E) \left(e_{i+1}(E) - e_{i}(E)\right)  \qquad \mbox{for $1 \leq i \leq \ell(E)-1$},
\end{align*} 
where $\lambda(E)$ is chosen so that the normalized spacings $\overline{\delta}_i(E)$ have
average length equal to $1$.  This allows one to ask whether the normalized spacings are described by a limiting law
and to ask how that law depends on the nature of the system (Hamiltonian $H$) being considered.   We have:

\vspace{0.1in}
\noindent
{\bf Conjecture of Berry and Tabor (1977) \cite{BT}:}\\
If the classical dynamical system associated to the Hamiltonian $H$ is integrable, then 
the limiting distribution of the normalized spacings for quantum energy levels is governed by a Poisson law.

\vspace{0.1in}
\noindent
{\bf Conjecture of Bohigas, Giannoni, and Schmit (1984) \cite{BGS}:}\\
If the classical dynamical system associated to the Hamiltonian $H$ is chaotic, then
the limiting distribution of the normalized spacings for the quantum energy levels is governed by the eigenvalue statistics of one of the
three standard ensembles of random matrices, GOE, GUE, or GSE.

\vspace{0.1in}

\noindent
For more information on these conjectures and various exciting extensions of them to other contexts we recommend \cite{IMA_BOOK,BDL}.

\vspace{0.1in}
As discussed in the paper of Berry and Tabor \cite{BT}, the multi-dimensional quantum harmonic
oscillator serves as a notable exception to these conjectures.  Even though the
classical multi-dimensional harmonic oscillator is completely integrable, the asymptotic behavior of the
normalized spacings distributions is not governed by a Poisson law because
the energy level contours are flat.  Rather, it depends dramatically on the
arithmetic properties of the ratio of frequencies.  A considerable portion of the paper by Berry and Tabor is devoted to this situation.  This was followed by the works of
 Dyson \cite{Dyson}, Boshernitzan \cite{Bosh1991,Bosh1992}, Bleher
\cite{B1,B2}, and others.

The multi-dimensional quantum harmonic oscillator is given by the Hamiltonian
\begin{align*}
H = -\displaystyle \sum_{j=1}^{d+1} \frac{\hbar ^2}{2m}\frac{\partial^2}{\partial x_j ^2} + \sum_{j=1}^{d+1} \frac{k_j}{2} x_j ^2.
\end{align*}
Applying Schr\"{o}dinger's equation, the quantum energy levels of the
system are determined by $d+1$ non-negative integers, $m_0, \ldots, m_d$, and
they are of the form
\begin{equation*}
e = e_0 + m_0\alpha_0 + \ldots + m_d\alpha_d,
\end{equation*}
where $\alpha_0, \ldots, \alpha_d$ are positive real numbers depending on the spring constant ${\bm k}$ and the mass $m$.
We have
\begin{equation*}
e = e_0 + \alpha_0(m_0 + m_1 \omega_1 + \ldots  + m_d \omega_d),
\end{equation*}
where $\omega_i = \alpha_i/\alpha_0 > 0$.  
We will refer to $\omega_1,\ldots,\omega_d$ the {\em ratios of frequencies}.

It is convenient to set 
$\Delta E = \alpha_0$ and to perform a linear change of variables to the energy scale, 
letting $\mathcal{E} = (E-e_0)/\alpha_0$, $\varepsilon = (e-e_0)/\alpha_0$, and hence $\Delta \E = 1$.  Then, the problem
becomes to find spacings between each of the neighboring energy levels
\begin{equation*}
\varepsilon = m_0 + m_1 \omega_1 + \ldots  + m_d \omega_d
\end{equation*}
lying between a given $\mathcal{E}$ and $\mathcal{E} + 1$ as $\mathcal{E} \rightarrow
\infty$.   All of our results will be described under this change of variables.

\vspace{0.05in}

The case $d=1$ corresponds to a two-dimensional harmonic oscillator and it was
carefully studied by Bleher \cite{B1,B2}, using methods involving continued
fractions.  He proved in \cite[Theorem 1.5]{B2} that for a generic ratio of the
frequencies $\omega \equiv \omega_1$ there is no limiting distribution of
normalized spacings.  Meanwhile in the special case that the ratio of
frequencies is the golden mean $\omega = (\sqrt{5}-1)/2$, Bleher explicitly
describes how the distribution of normalized spacings depends on $\E$; see
\cite[Theorem 3]{B1}.  From these formulae one can see that the distribution
behaves periodically in $\log \E$, asymptotically as $\E \rightarrow \infty$.
In other words, although there is no limit of the distribution of normalized
spacings, the asymptotic behavior remains rather simple.

Later it was proved by Bleher-Homma-Ji-Roeder-Shen in \cite[Theorem
1.6]{BlehHommJiRoedShen2012} that this phenomenon carries over to the three-dimensional
quantum harmonic oscillator when the ratios of frequencies satisfy that $1,
\omega_1, \omega_2$ form a basis for a cubic algebraic number field $\Phi$
that has one fundamental unit.  Rather than using continued fractions, which
do not generalize to higher dimensions, basic techniques from algebraic number
theory were used to prove that the distribution of normalized spacings behaves
asymptotically quasi-periodically in $\log \E$ as $\E \rightarrow \infty$.

In this paper we build on these previous results to prove:

\begin{THMA}
Suppose that the ratios of frequencies satisfy that $1,\omega_1,\ldots
,\omega_d\in\RR$ form a $\QQ$-basis for an algebraic number field $\Phi$.
We then have:

\begin{itemize}
\item[(1)] {\bf Uniform Labeling:}
There is a finite set 
$\S := \{s_1,\ldots,s_J\}  \subset \Phi$
and a rescaling factor $u({\mathcal E}) > 1$
 such that for any $\mathcal{E} > 1$ the rescaled spacings satisfy:
\begin{align*}
\overline{\delta_i}({\mathcal E}) := u({\mathcal E})(\varepsilon_{i+1}(\E) - \varepsilon_{i}(\mathcal{E})) \in \S \qquad \mbox{for $1 \leq i \leq \ell({ \mathcal E})-1$}.
\end{align*}

\item[(2)]
{\bf Quasiperiodicity:}
For each $1 \leq j \leq J$ let $p_j(\mathcal{E})$ denote the proportion of the rescaled spacings $\overline{\delta_i}({\mathcal E})$ such that $\overline{\delta_i}({\mathcal E}) = s_j \in \S$.  Then, 
there are integers $0 \leq k,r \leq d+1$, a Lipschitz continuous function
\begin{align*}
g: \mathbb{T}^k \times [0,1]^r \rightarrow P := \left\{(p_1,\ldots,p_J) \, : \sum p_j = 1, p_j \geq 0\right\},
\end{align*}
angles ${\bm \theta} = (\theta_1,\ldots,\theta_k) \in \mathbb{T}^k$, rates ${\bm \beta} = (\beta_1,\ldots,\beta_r) \in \mathbb{R}^r$, and $0 < \alpha < 1$ such that
\begin{align*}
(p_1({\mathcal E}),\ldots,p_J({\mathcal E})) = g\left({\bm \theta} \log {\mathcal E}, \{ {\bm \beta} \log {\mathcal E} \} \right) + \mathcal{O}(\alpha^{\log {\mathcal E}}).
\end{align*}
Here, we use the notation
\begin{align*}
\{ {\bm \beta} \log \E \} := \left(\{ \beta_1 \log \E \}, \ldots, \{ \beta_r \log \E \} \right) \in [0,1]^r,
\end{align*}
where $\{x\}$ denotes the fractional part of a real number $x$.
\end{itemize}
\end{THMA}

\begin{remark} Note that $\E \mapsto \left({\bm \theta} \log {\mathcal E}, \{ {\bm \beta} \log {\mathcal E} \} \right)$
describes a linear flow in $\log \E$ time on the ``generalized annulus'' $\mathbb{T}^k \times [0,1]^r$, i.e.\ a quasiperiodic motion.  This
is why we say that the rescaled spacings distribution depends quasiperiodically on $\log \E$, asymptotically as $\E \rightarrow \infty$.
\end{remark}

\begin{remark}\label{RMK_LABELLING_AMBIGUITY}  Note also that the rescaling factor $u({\mathcal E})$ is adapted
to the number field $\Phi$ and therefore it might not precisely equal the
normalization factor $\lambda(\E)$.   I.e.\ the rescaled spacings may not be
normalized to have average length exactly equal to $1$.  This is why we refer
to $u({\mathcal E})$ as the ``rescaling factor'' rather than the ``normalizing
factor''.  
\end{remark}

For many quantum systems it is difficult to determine the normalization factor $\lambda(E)$.
In 2007 it was proposed by Oganesyan and Huse \cite{OH} that instead of studying the statistics of the normalized spacings $\overline{\delta}_i(E)$ for $2 \leq i \leq \ell(\E)$ one can study the ratios of neighboring spacings
\begin{align*}
\rho_i(\E) := \frac{\delta_i(\E)}{\delta_{i-1}(\E)} \qquad  \mbox{for $2 \leq i \leq \ell(E)-1$}.
\end{align*}
(Note that in \cite{OH} they take the reciprocal of $\rho_i(E)$ in the case that $\rho_i(E) > 1$, but we will not do that.)
The conjectures of Berry and Tabor and of Bohigas, Giannoni, and Schmit can then be re-phrased in terms of
ratios.  See, for example, \cite{ABGR}.

Considering the ratios of neighboring spacings is also quite suitable in our
setting because the ambiguity associated to the choice of rescaling factor
$u(\E)$ that is described in Remark \ref{RMK_LABELLING_AMBIGUITY} is
eliminated.  We have:

\begin{THMB}
Suppose that the ratios of frequencies satisfy that $1,\omega_1,\ldots
,\omega_d\in\RR$ form a $\QQ$-basis for an algebraic number field $\Phi$.
We then have:

\begin{itemize}
\item[(1)] {\bf Uniform Set of Ratios:}
There is a finite set
$\R := \{r_1,\ldots,r_J\}  \subset \Phi$
such that for any $\mathcal{E} > 1$ the ratios of neighboring spacings satisfy:
\begin{align*}
\rho_i(\E) := \frac{\delta_i(\E)}{\delta_{i-1}(\E)} \in \R \qquad \mbox{for $2 \leq i \leq \ell({ \mathcal E})-1$}.
\end{align*}

\item[(2)]
{\bf Quasiperiodicity:}
For each $1 \leq j \leq J$ let $p_j(\mathcal{E})$ denote the proportion of the ratios of {\em neighboring} spacings $\rho_i({\mathcal E})$ such that $\rho_i({\mathcal E}) = r_j \in \R$.  Then,
there are integers \hbox{$0 \leq k,r \leq d+1$,} a Lipschitz continuous function
\begin{align*}
h: \mathbb{T}^k \times [0,1]^r \rightarrow P := \left\{(p_1,\ldots,p_J) \, : \sum p_j = 1, p_j \geq 0\right\},
\end{align*}
angles ${\bm \theta} = (\theta_1,\ldots,\theta_k) \in \mathbb{T}^k$, rates ${\bm \beta} = (\beta_1,\ldots,\beta_r) \in \mathbb{R}^r$, and $0 < \alpha < 1$ such that
\begin{align*}
(p_1({\mathcal E}),\ldots,p_J({\mathcal E})) = h\left({\bm \theta} \log {\mathcal E}, \{ {\bm \beta} \log {\mathcal E} \} \right) + \mathcal{O}(\alpha^{\log {\mathcal E}}).
\end{align*}
\end{itemize}
\end{THMB}

\begin{remark}
Theorem B does not immediately follow from Theorem A since the ratios of neighboring spacings are considered
 rather than the ratios between arbitrary pairs of spacings.   
However, in Section~\ref{SEC:QP_ON_WORDS_AND_PROOF_THM_BPRIME} we will prove a generalization of Theorem A which implies Theorem B.
\end{remark}

In order to prove Theorems A and B we need a way to compute all of the energy
levels $\varepsilon$ occurring in
the window $\E \leq \varepsilon < \E + 1$
 for a give choice of $\E > 0$.
For any vector ${\bm m} = (m_1,\ldots, m_d)$ of non-negative integers
such that ${\bm m}\cdot {\bm \omega} < \mathcal{E} + 1$, there is
exactly one integer $m_0 \ge 0$ that forces 
\begin{align*}
 \E \leq \varepsilon := m_0 + m_1 \omega_1 + \ldots  + m_d < \E + 1.
\end{align*}
This allows us to reduce the problem modulo 1,
considering differences between the fractional parts of the numbers $m_1
\omega_1 + \ldots + m_d\omega_d$ determined by integer vectors ${\bm m} \in
R(t)$. Here, $R(t)$ is the homothetic expansion of \begin{align*} R = \{{\bm v}
\in \mathbb{R}^n \, : \,  v_1\omega_1 + \ldots + v_d\omega_d < 1 \quad \mbox{and} \quad v_i \geq 0 \quad \mbox{for $1 \leq i \leq d$}\}
\end{align*} by a factor of $t=\mathcal{E}+1$ about the origin.   
Therefore, Theorems A and B will be special cases of Theorems A' and B' that are stated in the next
subsection.

\subsection{Mathematical Perspective}\label{SEC:MATHEMATICAL_PROBLEM}
Having provided the physical context and statements of our results,
we will now rephrase them purely in the mathematical context.

Suppose that $d\in\NN$ and that $\bm{\omega}=(\omega_1,\ldots ,\omega_d)\in\RR^d$. Let $R$ be a bounded, convex region in $\RR^d$ with non-empty interior, and for $t\ge 1$ let $R(t)$ denote the homothetic dilation of $R$ by a factor of $t$.  Throughout the paper we will refer to $t \ge 1$ as the ``scale''.  Let $\{x\}$ denote the fractional part of a real number $x$, set $M(t)=R(t)\cap \ZZ^d$, and write the elements of the set
\begin{equation*}
Y(t):=\left\{\{\bm{m}\cdot\bm{\omega }\}:\bm{m}\in M(t)\right\}
\end{equation*}
in order as 
\begin{equation*}
0\le y_1(t)\le\cdots \le y_{|M(t)|}(t)<1.
\end{equation*}
For each value of $i=1,\ldots ,|M(t)|-1$, let
\begin{equation*}
\delta_i(t)=y_{i+1}(t)-y_{i}(t),
\end{equation*}
and let $D(t)$ be the number of distinct elements of the set $\{\delta_i(t)\}_{i=1}^{|M(t)|-1}$. Finally, let
\begin{equation*}
\Delta_1(t)<\cdots <\Delta_{D(t)}(t)
\end{equation*}
be the ordered sequence of these distinct elements. We may also write $\Delta_i$ for $\Delta_i(t)$ and, following~\cite{BlehHommJiRoedShen2012}, we refer to the quantities $\delta_i$ and $\Delta_i$ as \textit{spacings}.

The classical Three Gap Theorem (also called the Steinhaus Conjecture and the Three Distance Theorem) states that if $d=1$, then $D(t)\le 3$ for all $t$. This was first proved  in 1957 by S\'{o}s \cite{Sos1957}, in 1958 by Sur\'{a}nyi \cite{Sura1958}, and in 1959 by \'{S}wierczkowski \cite{Swie1959}. In the case when $d\ge 2$, estimating the size of $D(t)$ is a more difficult problem. It was known to Geelen and Simpson (attributed by them to Holzman in \cite[Section 4]{GeelSimp1993}) that, in the case when $d=2,$ if $R$ is a square with sides parallel to the coordinate axes and if $1,\omega_1,$ and $\omega_2$ are $\QQ$-linearly dependent, then
\begin{equation}\label{eqn.DBdd}
\sup_{t\ge 1}D(t)<\infty.
\end{equation}
A proof of this, as well as an extension to $d\ge 3$, is given in \cite[Section
4]{HaynMark2020}. A problem attributed to Erd\H{o}s is to determine for what
values of $\bm{\omega}$ the quantity $D(t)$ remains bounded. It was first
speculated that the condition that the numbers $1, \omega_1,\ldots , \omega_d$
be $\QQ$-linearly dependent is necessary in order for $D(t)$ to remain bounded.
However, the situation is more subtle.

A vector ${\bm \omega} \in \mathbb{R}^d$ is called \emph{Diophantine} with exponent $\gamma$ if there is some positive number $K$ such that for all nonzero vectors ${\bm m} \in \mathbb{Z}^d$, we have
\begin{equation}\label{EQ:BA}
|{\bm m} \cdot {\bm \omega}| \geq \frac{K}{|{\bm m}|^\gamma}, \textrm{ where } |{\bm m}| = \sqrt{m_1^2 + \ldots + m_d^2}.
\end{equation}
It follows from the Minkowski's Theorem that $\gamma \geq d$. We call ${\bm \omega}$ \emph{badly approximable} if $\gamma = d$.  The set of badly approximable
$\bm{\omega}$ has zero Lebesgue measure, but it is known by work of Jarn\'{i}k \cite{Jarn1928} and Wolfgang Schmidt \cite{Schm1969} to be a subset of $\RR^d$ of Hausdorff dimension $d$.

\begin{THM_BD} 
If ${\bm \omega} \in \mathbb{R}^d$ is badly approximable, then (\ref{eqn.DBdd}) holds.
\end{THM_BD}

\noindent
This result was not published by Boshernitzan and Dyson, but a proof can be found in \cite{BlehHommJiRoedShen2012}.

In the other direction, it was proved in \cite{HaynMark2020} that when $d\ge 2$, for almost all choices of $\bm{\omega}\in\RR^d$,
\begin{equation}\label{eqn.DUnbdd}
\sup_{t\ge 1}D(t)=\infty.
\end{equation}
The proof given in \cite{HaynMark2020} uses ergodic theory in spaces of unimodular lattices in $\RR^{d+1}$. Part of the interest in this problem lies in the fact that, for $d\ge 2$, if \eqref{eqn.DUnbdd} holds, then
\begin{equation*}
\liminf_{n\rar\infty}n\|n\omega_1\|\cdots\|n\omega_d\|=0,
\end{equation*}
where $\|\cdot\|$ denotes the distance to the nearest integer \cite[Theorem 3]{HaynMark2020}. In other words, if \eqref{eqn.DUnbdd} holds then the generalized Littlewood conjecture is true for $\bm{\omega}$. The converse of this statement, however, is not true.

We also remark that, as another twist in the above mentioned problem of Erd\H{o}s, it turns out that there do exist vectors ${\bm \omega} \in \mathbb{R}^d$ that are not badly approximable but for which (\ref{eqn.DBdd}) holds.  An explicit example is given in \cite{BK} for $d=2$.

Now let us focus on the situation when $1,\omega_1,\ldots
,\omega_d\in\RR$ form a $\QQ$-basis for an algebraic number field $\Phi$ of
degree $d+1$. In this case it is known by work of Perron \cite{Perr1921} that
$\bm{\omega}$ is badly approximable, so it follows from the Boshernitzan-Dyson Theorem
that \eqref{eqn.DBdd} holds.  (In fact, Dyson originally proved the theorem in this special case.)  Actually, more is true, as demonstrated by
the following result.

\begin{theorem}\cite[Theorem 1.6]{BlehHommJiRoedShen2012}\label{thm.BHJRS1.6}
If $1,\omega_1,\ldots ,\omega_d\in\RR$ form a $\QQ$-basis for an algebraic number field $\Phi$, then there exists a finite set 
\begin{equation*}
\S=\{s_1, \cdots,s_J\}\subseteq\Phi
\end{equation*}
such that every spacing $\Delta_i$ has the form $us_j$ for some unit $u$ in the ring of integers $\ZZ_\Phi$ and some $s_j\in \S$.
\end{theorem}

\noindent
Throughout this paper we will denote by $\ZZ_\Phi^\times$ the group of units in the ring of integers $\ZZ_\Phi$ of an algebraic number field $\Phi$.

Consider the one dimensional case $d=1$.  As $t \ge 1$ increases, additional points
$y_j(t)$ are added one-by-one.  Most of the time, this results in one of the
large-sized spacings being split into one of the mid-sized spacings and one of
the small-sized spacings.  This continues until each of the large-sized
spacings is split, at which point one renames the medium-sized spacings to be ``large'' and the small-sized spacings to be ``medium''
and then repeats the process, splitting each new large-sized spacing into a new medium-sized spacing and a (truly new) small-sized spacing.
Therefore, as $t$ increases, the
proportions of spacings occurring that are deemed to be ``small'', ``medium'', and ``large'' depend
in an organized way on $t$, which can be made precise using the theory of continued fractions.

The purpose of this paper is to describe analogous behavior in the far more complicated situation 
when $d \geq 2$, at least in the algebraic case when $1,\omega_1,\ldots ,\omega_d\in\RR$ form a $\QQ$-basis for an algebraic number field $\Phi$.
We will prove in Part (1)~of Theorem~A', below,  a stronger version of Theorem
\ref{thm.BHJRS1.6} that allows us to \textit{uniformly} label the spacings $\Delta_1(t),\ldots,\Delta_{D(t)}(t)$
using the elements of a finite set $\S \subset \Phi$.  We can then prove in Part (2) of Theorem~A' a description
of the time $t$
evolution of the proportion of spacings realizing these labels from $\S$.  
We will also prove a similar theorem about the frequencies with which the ratios of neighboring spacings
occur, see Theorem~B' below.

\begin{THMAPRIME}
Suppose that $1,\omega_1,\ldots
,\omega_d\in\RR$ form a $\QQ$-basis for an algebraic number field $\Phi$
and let $r \geq 1$ denote the rank of $\ZZ_\Phi^\times$.
We then have:

\begin{itemize}
\item[(1)] {\bf Uniform Labeling:}
There is a finite set
$\S := \{s_1,\ldots,s_J\}  \subset \Phi$
and a rescaling factor $u(t) \in \ZZ_\Phi^\times$
 such that for any $t > 1$ we have
\begin{align*}
\overline{\delta_i}(t) := u(t)\delta_i(t) \in \S \qquad \mbox{for $1 \leq i \leq |M(t)|-1$}.
\end{align*}
We will refer to the $\left\{\overline{\delta_i}(t)\right\}$ as the {\em rescaled spacings}.

\item[(2)]
{\bf Quasiperiodicity:}
For each $1 \leq j \leq J$ let $p_j(t)$ denote the proportion of the rescaled spacings $\overline{\delta_i}(t)$ such that $\overline{\delta_i}(t) = s_j \in \S$.  Then,
there is an integer $0 \leq k \leq d+1$, a Lipschitz continuous function
\begin{align*}
g: \mathbb{T}^k \times [0,1]^r \rightarrow P := \left\{(p_1,\ldots,p_J) \, : \sum p_j = 1, p_j \geq 0\right\},
\end{align*}
angles ${\bm \theta} = (\theta_1,\ldots,\theta_k) \in \mathbb{T}^k$, rates ${\bm \beta} = (\beta_1,\ldots,\beta_r) \in \mathbb{R}^r$, and $0 < \alpha < 1$ such that
\begin{align*}
(p_1(t),\ldots,p_J(t)) = g\left({\bm \theta} \log t, \{ {\bm \beta} \log t \} \right) + \mathcal{O}(\alpha^{\log {
t}}).
\end{align*}
Here, we use the notation
\begin{align*}
\{ {\bm \beta} \log t \} := \left(\{ \beta_1 \log t \}, \ldots, \{ \beta_r \log t \} \right) \in [0,1]^r,
\end{align*}
where $\{x\}$ denotes the fractional part of a real number $x$.
\end{itemize}
\end{THMAPRIME}

Note that $t \mapsto \left({\bm \theta} \log t, \{ {\bm \beta} \log
t \} \right)$ describes a linear flow in $\log t$ time on the
``generalized annulus'' $\mathbb{T}^k \times [0,1]^r$, i.e.\ a quasiperiodic
motion.  Therefore, 
Theorem~A' asserts that the frequencies at which the rescaled spacings $s_1
< s_2 < \cdots < s_J$ occur, as a function of $t$, depend quasiperiodically on
$\log t$, as $t\rar\infty$. 
To understand some of the details of this
dependence, before delving into the complexities of the proof, first note there
are many ways to choose a basis $\{\epsilon_1,\ldots ,\epsilon_r\}$ for a
finite index subgroup of $\ZZ_\Phi^\times$, none of which in general can be
assumed to be canonical. All of the parameters in the statement of the theorem
depend on this choice of basis and, once it has been made, they are explicitly
computable. The units $u(t)$ will turn out to be determined by
\[u(t)^{-1}=\epsilon_1^{\lfloor \beta_1 \log t\rfloor} \cdots \epsilon_r^{\lfloor
\beta_r \log t\rfloor},\] for a suitable choice of ${\bm \beta}$, and this in
turn determines $\S$. The Minkowski embedding of $\Phi$ into $\RR^{d+1}$ (see
next section for details) allows us to view multiplication by $u(t)^{-1}$ in $\Phi$
as a linear transformation $U(t)$ on $\RR^{d+1}$. By using the Jordan
decomposition of a complex matrix defined using a particular choice of matrix
logarithm for $U(t)$, we are able to understand the time evolution of $U(t)$ as
$t\rar\infty$, and to prove that it is governed completely by two types of
generalized eigenspaces: one dimensional spaces corresponding to purely
imaginary eigenvalues $2\pi i\theta_1,\ldots , 2\pi i \theta_k$ (of which there
is at least one), and spaces with eigenvalues whose real parts are negative.
This gives $k$ and ${\bm \theta}$, and $\alpha$ is determined by the maximum of
the real parts of the eigenvalues which are not purely imaginary (if there are
any). Finally, the function $g$ is defined explicitly in Section
\ref{subsec.glue} and equation \eqref{eqn.g3def}, it is straightforward to
compute and depends only on $\S$ and the region $R$.  

\begin{remark} Under the
additional hypothesis that $\Phi$ has one fundamental unit (which restricts the theorem to quadratic fields and cubic fields with a complex embedding), a preliminary version of Theorem~A' was proved in
\cite[Thms 1.5 and 1.6]{BlehHommJiRoedShen2012}.  The term involving rate
$\beta \equiv \beta_1$ does not appear in those theorems because they are
expressed at a sequence of times $t_n = \eta {\rm e}^{n/\beta}$ that is chosen
to make the sequence $\{\beta \log t_n\}$ constant.  When $r > 1$ such a choice
is not typically possible. However, several aspects of the proofs from
\cite[Thms 1.5 and 1.6]{BlehHommJiRoedShen2012} will play an important role in
our proof of Theorem~A'.  \end{remark}

The labeling of spacings by elements of $\S$ may seem ad
hoc because it depends on the construction of $u(t)$ from the proof of Theorem
A', Part (1).  By adjusting the choice of $u(t)$ we could easily change the set of labels $\S$. 
For this reason, it may be more natural to consider the ratios between the neighboring spacings.

\begin{THMBPRIME}
Suppose that $1,\omega_1,\ldots
,\omega_d\in\RR$ form a $\QQ$-basis for an algebraic number field $\Phi$
and let $r \geq 1$ denote the rank of $\ZZ_\Phi^\times$.
We then have:

\begin{itemize}
\item[(1)] {\bf Uniform Set of Ratios:}
There is a finite set
$\R := \{r_1,\ldots,r_J\}  \subset \Phi$
such that for any $t > 1$ the ratios of neighboring spacings satisfy:
\begin{align*}\label{EQN:FINITE_RATIO_SET}
\rho_i(t) := \frac{\delta_i(t)}{\delta_{i-1}(t)} \in \R \qquad \mbox{for $2 \leq i \leq |M(t)|-1$}.
\end{align*}

\item[(2)]
{\bf Quasiperiodicity:}
For each $1 \leq j \leq J$ let $p_j(t)$ denote the proportion of the ratios of neighboring spacings $\rho_i(t)$ such that $\rho_i(t) = r_j \in \R$.  Then,
there is an integer $0 \leq k \leq d+1$, a Lipschitz continuous function
\begin{align*}
h: \mathbb{T}^k \times [0,1]^r \rightarrow P := \left\{(p_1,\ldots,p_J) \, : \sum p_j = 1, p_j \geq 0\right\},
\end{align*}
angles ${\bm \theta} = (\theta_1,\ldots,\theta_k) \in \mathbb{T}^k$, rates ${\bm \beta} = (\beta_1,\ldots,\beta_r) \in \mathbb{R}^r$, and $0 < \alpha < 1$ such that
\begin{align*}
(p_1(t),\ldots,p_J(t)) = h\left({\bm \theta} \log t, \{ {\bm \beta} \log t \} \right) + 
\mathcal{O}(\alpha^{\log t}).
\end{align*}
\end{itemize}
\end{THMBPRIME}

\begin{remark}\label{REMARK_ABOUT_THMC}
Theorem B' does not immediately follow from Theorem A' because only the ratios
of neighboring spacings are considered, rather than the ratios between
arbitrary pairs of spacings.   However, in Section \ref{SEC:QP_ON_WORDS_AND_PROOF_THM_BPRIME} we will state and
prove a stronger version of Theorem A', namely Theorem~C, about the frequencies with which a given
word $s_{j_1} s_{j_2} \ldots s_{j_l}$ appears among the consecutive
rescaled spacings.   Theorem B' then follows immediately from Theorem C using words of length $l = 2$.
\end{remark}

\begin{question}
Let ${\bm \omega} \in \mathbb{R}^d$ and suppose Statement {\rm (1)} from Theorem B' holds
for the ratios of spacings determined by  ${\bm \omega}$.
Must $1, \omega_1, \ldots, \omega_d$ be a $\mathbb{Q}$-basis for an algebraic number field~$\Phi$?
\end{question}

\subsection{Plan for the paper}

As explained at the end of Section \ref{SUBSEC:PHYSICAL}, Theorems A and B are
direct consequences of Theorems A' and B'. 

Our proof of Part (1) of Theorem~A' is an application of transference
principles from Diophantine approximation, together with well known results
from algebraic number theory. In Section \ref{sec.Prelim} we will review some
of these results, and in Section \ref{sec.PfSteinFinite} we will present the
proof of Part (1) of Theorem~A'.
Our proof of Part (2) of Theorem~A' uses many results from the proofs of 
\cite[Thms 1.5 and 1.6]{BlehHommJiRoedShen2012} combined with several new ideas that are needed
when $\mathbb{Z}_\Phi^\times$ has rank $r > 1$.  It is presented in Section~\ref{sec.QP}.
In Section \ref{SEC:QP_ON_WORDS_AND_PROOF_THM_BPRIME} we prove Theorem C, the generalization of
Theorem A' mentioned in Remark \ref{REMARK_ABOUT_THMC}.  Since Theorem B' is an immediate corollary to Theorem C
this will also complete the proof of Theorem B'.
Finally, in Section \ref{sec.example} we work out the details of Theorems A' for a particular example (a totally real cubic field) which highlights the computational aspects and importance of many of the steps in our proofs.

\subsection{Acknowledgments}
The second author thanks Pavel Bleher for introducing him to this subject and
for many interesting conversations about it.  We thank Evgeny Mukhin and Vitaly
Tarasov for providing us with the proof of Lemma \ref{LEM:COMMUTING_LOGS},
which plays a crucial role in our paper.  The second author also thanks Aneesh Dasgupta for interesting conversations about this subject. The work of the first author was
supported by NSF grant DMS-2001248. The work of the second author was
supported by NSF grant DMS-1348589.


\section{Notation and preliminary results}\label{sec.Prelim}
For $x\in\RR$, we write $\{x\}$ for the fractional part of $x$ and $\|x\|$ for the distance from $x$ to the nearest integer. For $d\in\NN$ and $\bm{x}\in\RR^d$, we write $|\bm{x}|$ for the Euclidean norm of $\bm{x}$.

Results from Diophantine approximation known as transference principles (see \cite[Section V, Theorem VI]{Cass1957})  imply that $\bm{\omega}\in\RR^d$ is badly approximable if and only if there exists a constant $K'>0$ with the property that, for any $t\ge 1$, any ball of diameter $K'/t^d$ in $[0,1)$ contains a point of the set
\begin{equation*}
\left\{\{\bm{m}\cdot\bm{\omega }\}:\bm{m}\in M(t)\right\}.
\end{equation*}
If $1,\omega_1,\ldots ,\omega_d$ form a $\QQ$-basis for an algebraic number field of degree $d+1$ over $\QQ$,  then the results from \cite{Perr1921} mentioned in the introduction imply that $\bm{\omega}$ is badly approximable. By the transference principle just cited, we thus have for all $t\ge 1$ and for all $i$ that
\begin{equation}\label{eqn.TransBd}
\Delta_i\le\frac{K'}{t^d}.
\end{equation}

Next we summarize some basic facts from algebraic number theory, proofs of which can be found in \cite[Chapters 1, 3]{Swin2001}. As above, suppose that $\Phi$ is an algebraic number field of degree $d+1$ over $\QQ$. There are $d+1$ distinct embeddings of $\Phi$ into $\CC$, and the non-real complex embeddings come in complex conjugate pairs. Suppose there are $r_1$ real embeddings and $2r_2$ complex embeddings, and write $\sigma_1,\ldots ,\sigma_{r_1}$ for the real embeddings and $\sigma_i,\sigma_{i+r_2}$, with $r_1<i\le r_1+r_2$, for each pair of complex conjugate embeddings. Identifying $\CC$ with $\RR^2$, we define a map $\sigma:\Phi\rar\RR^{d+1}$ by
\begin{equation*}
\sigma(\alpha) = (\sigma_1(\alpha),\ldots ,\sigma_{r_1+r_2}(\alpha)).
\end{equation*}
This map is injective, and the set
\begin{equation*}
\Gamma=\sigma(\ZZ_\Phi)
\end{equation*}
is called the Minkowski embedding of the ring of integers of $\Phi$ into $\RR^{d+1}$. It is a discrete subgroup of $\RR^{d+1}$, and the quotient $\RR^{d+1}/\Gamma$ has a measurable fundamental domain of finite volume. In other words, $\Gamma$ is a lattice in $\RR^{d+1}$. 

Let $k \in \ZZ$ be chosen so that for each $1 \leq j \leq d$ we have $k  \omega_j \in \ZZ_\Phi$.  Notice that
\begin{align}\label{eqn.MPhi}
\mathbb{M}_\Phi := \{\alpha \in \Phi \, : \, k  \alpha \in \ZZ_\Phi\}
\end{align}
is a $\ZZ_\Phi$ module that contains $\omega_1,\ldots,\omega_d$.  In particular, for any $t \ge 1$ each spacing $\Delta_i(t) \in \mathbb{M}_\Phi$ and moreover
$x \Delta_i(t) \in \mathbb{M}_\Phi$ for any $x \in \ZZ_\Phi$.  The image
\begin{align}\label{eqn.ZModLat}
\Gamma'=\sigma(\mathbb{M}_\Phi)
\end{align}
is a lattice because $\Gamma' = \frac{1}{k} \Gamma$, with $\Gamma$ a lattice.

Next, let $\ZZ_\Phi^\times$ denote the group of multiplicative units of $\ZZ_\Phi$. By the Dirichlet unit theorem, this group has rank $r_1+r_2-1$. Consider the map $\varphi:\ZZ_\Phi^\times\rar\RR^{r_1+r_2}$ defined by
\begin{equation}\label{EQN:DEF_VARPHI}
\varphi (u)=\left(\log\left|\sigma_1(u)\right|,\ldots ,\log\left|\sigma_{r_1+r_2}(u)\right|\right).
\end{equation}
This map is well defined, since $|\sigma_i(u)|\not= 0$. The norm of any unit is $\pm 1$, so the image of $\varphi$ is contained in the hyperplane in $\RR^{r_1+r_2}$ with equation
\begin{equation}\label{eqn.Hyper}
x_1+\cdots +x_{r_1}+2x_{r_1+1}+\cdots +2x_{r_1+r_2}=0.
\end{equation} 
Furthermore, the image of $\varphi$ is a lattice in this hyperplane (see the proof of \cite[Theorem 11]{Swin2001}). All of these facts will be useful to us in what follows.

\section{Proof of Part (1) of  Theorem A' (Uniform Labeling)}\label{sec.PfSteinFinite}
Suppose without loss of generality that $t\ge 1.$ Since $\bm{\omega}\in\RR^d$ we know that $r_1\ge 1$, so let us assume that $\sigma_1:\Phi\rar\RR$ is the trivial embedding which maps each number $\omega_j$ to itself. Each spacing $\Delta_i(t)$ has the form
\begin{equation*}
\Delta_i(t)=(\bm{m}-\bm{m}')\cdot\bm{\omega},
\end{equation*}
for some $\bm{m},\bm{m}'\in M(t)$. It follows from this and \eqref{eqn.TransBd} that there exists a constant $C>0$, which does not depend on $t$, with the property that for each spacing $\Delta_i$, we have that
\begin{equation*}
|\sigma_1(\Delta_i)|\le \frac{C}{t^d}.
\end{equation*}
Also, since $|\bm{m}-\bm{m}'|$ is bounded by a constant times $t$, and the maps $\sigma_j$ are homomorphisms, it immediately follows that 
\begin{equation*}
 |\sigma_j(\Delta_i)|\le Ct.
\end{equation*}
for $2\le j\le r_1+r_2$.

Since $\varphi (\ZZ_\Phi^\times)$ is a lattice in the hyperplane defined by \eqref{eqn.Hyper}, there is a constant $C'>0$ with the property that, for any point $\bm{x}\in\RR^{r_1+r_2}$ satisfying \eqref{eqn.Hyper}, there is an element of $\varphi (\ZZ_\Phi^\times)$ in the ball of radius $C'$ centered at $\bm{x}$. Using the fact that $r_1+2r_2=d+1$, we apply this observation with
\begin{equation}\label{EQN:X_OF_T}
\bm{x}=(d\log t,-\log t,\ldots ,-\log t).
\end{equation}
We thus deduce that there is a unit $u=u(t)$ with the properties that
\begin{equation*}
|\sigma_1(u)|\le e^{C'}t^d
\end{equation*}
and, for $2\le j\le r_1+r_2$, that
\begin{equation*}
|\sigma_j(u)|\le \frac{e^{C'}}{t}.
\end{equation*}
It follows that $\sigma(u \Delta_i)$ is a point of the set $\Gamma'$ from \eqref{eqn.ZModLat}, which lies in a cube $\mc{C}$ of side length $2Ce^{C'}$ centered at the origin in $\RR^{d+1}$. Let $\S$ be the collection of all elements of $\mathbb{M}_\Phi$  (defined in (\ref{eqn.MPhi})) 
whose images under $\sigma$ lie in $\mc{C}$. Since $\Gamma'$ is a lattice and $\sigma$ is injective, the set $\S$ is finite. The statement of the theorem thus follows.  \qed (Part (1) of Theorem A')


\section{Proof of Part (2) of Theorem A' (Quasiperiodicity)}\label{sec.QP}

\noindent
Several times in this section we will need to refer to the muliplicative inverse of $u(t)$ and 
thus we will 
denote it by $u_1 \equiv u_1(t):= u(t)^{-1}$.

\vspace{0.1in}

Let ${\bm n}:  \Phi \rightarrow \mathbb{Q}^{d+1}$
denote the expansion of an element of $\Phi$ in terms of the basis $1,\omega_1,\ldots,\omega_d$.  That is, for any $\alpha \in \Phi$, 
\begin{align*}
{\bm n}(\alpha) = (n_0,n_1,\ldots,n_d) \qquad \mbox{iff} \qquad \alpha = n_0 + n_1 \omega_1 + \ldots + n_d \omega_d.
\end{align*}
Let ${\bm m}:  \Phi \rightarrow \mathbb{Q}^{d}$ be the ``truncated expansion'' of $\alpha$ given by
\begin{align*}
{\bm m}(\alpha) = (n_1,\ldots,n_d) \qquad \mbox{if} \qquad {\bm n}(\alpha) = (n_0,n_1,\ldots,n_d).
\end{align*}

The preliminary versions of Part (2) of Theorem A' from \cite{BlehHommJiRoedShen2012} are proved in three steps:
\begin{enumerate}
\item  Describing the proportions of spacings at a given scale $t\ge 1$ in terms of suitable partitions of $M(t)$ and $R$, 
\item  Relating these partitions to $\frac{{\bm n}(u_1(t))}{t}$, the normalized expansion of $u_1(t)$ in the basis $1,\omega_1,\ldots,\omega_d$, and
\item  Analysis of the asymptotic behavior of $\frac{{\bm n}(u_1(t))}{t}$ as $t \rightarrow \infty$.
\end{enumerate}
Steps (1) and (2) carry over directly to our setting.  We will describe them in
Sections \ref{subsec.Partitions} and~\ref{subsec.glue}, with a discussion of any necessary adaptations.
However, Step (3) requires some new ideas, which we present in Section \ref{subsec.dynamics},
thus completing the proof of Part (2) of Theorem A'.
(Section \ref{subsec.matrix} presents some lemmas that are needed in Sections \ref{subsec.glue} and \ref{subsec.dynamics}.)

\subsection{Partitions of $M(t)$ and $R$.}\label{subsec.Partitions}
For any $1 \leq j \leq J$, let $Y_j(t)$ be the set of numbers $y_i(t)$ such that $\delta_i(t) = y_{i+1}(t) - y_{i}(t) = s_j u_1(t)$.  Here, $u_1(t) = u(t)^{-1}$, where $u(t)$
is the unit from Part (1) of Theorem~A' and $s_j$ is an element of the finite set $\S$, which has been ordered so that $s_1<\cdots <s_J$.

Let $M_j(t)$ be the set of vectors ${\bm m} \in M(t)$ such that $\{{\bm m} \cdot {\bm \omega}\} \in Y_j(t)$.  Up to the single point corresponding to the largest element $y_{\ell(t)}(t) \in Y(t)$ we have that
\begin{align*}
\bigsqcup_{j=1}^J M_j(t) = M(t),
\end{align*}
so, by a slight abuse of notation, we will call $\{M_j(t)\}$ a partition of $M(t)$.  We refer the reader to \cite[Fig. 1]{BlehHommJiRoedShen2012} for an explicit example.
We conclude that at scale $t \ge 1$ the proportion of times that the rescaled spacing $s_j \in \S$ occurs is 
\begin{align*}
p_j(t) = \frac{|M_j(t)|}{|M(t)|-1}.
\end{align*}

We will need the following result from \cite{BlehHommJiRoedShen2012}.
\begin{proposition}[Prop.\ 5.1 from \cite{BlehHommJiRoedShen2012}]\label{PROP:BLEHER_5_1}
For any $1 \leq j \leq J$ let 
\begin{align*}
{\bm v}_j(t) = {\bm m}(s_j u_1(t)).
\end{align*}
Then we have
\begin{align}\label{EQN:PARTITIONS}
M_j(t) = [M(t) \cap (M(t) - {\bm v}_j(t))] \setminus \bigcup_{i=1}^{j-1}(M(t) - {\bm v}_i(t)),
\end{align}
where $\Omega + {\bm v}$ is defined to be $\{{\bm u} + {\bm v} \, : \, {\bm u} \in \Omega\}$ for any $\Omega \subset \mathbb{R}^d$ and ${\bm v} \in  \mathbb{R}^d$.
\end{proposition}

\begin{remark}
Our formula for $M_j(t)$ above has minus signs where the analogous formula
in \cite{BlehHommJiRoedShen2012} has plus signs.   The reason is that we define
$\delta_i(t) = y_{i+1}(t)-y_i(t)$ while $y_i(t)-y_{i-1}(t)$ is used in \cite{BlehHommJiRoedShen2012}.
\end{remark}

Denote the power set of our region $R$ by $\mathcal{P}(R)$ and let 
\begin{align*}
P : (\mathbb{R}^d)^J \rightarrow \mathcal{P}(R)^J
\end{align*}
be the mapping which sends the $J$-tuple of vectors ${\bm v} = ({\bm v_1},\ldots,{\bm v_J})$ to the $J$-tuple $(P_1({\bm v}),\ldots,P_J({\bm v}))$ of subsets of $R$ where,
\begin{align*}
P_j({\bm v}) = [R \cap (R - {\bm v_j})] \setminus \bigcup_{i=1}^{j-1}(R - {\bm v_i}).
\end{align*}
\begin{proposition}[Prop.\ 5.2 from \cite{BlehHommJiRoedShen2012}]\label{PROP:BLEHER_5_2}
If 
\begin{align}\label{EQN:DEF_V_T}
{\bm v} := {\bm v}(t) = \left(\frac{{\bm m}(s_1 u_1(t))}{t},\ldots,\frac{{\bm m}(s_J u_1(t))}{t}\right), 
\end{align}
then
\begin{align}
p_j(t) - \frac{{\rm volume}(P_j({\bm v}))}{{\rm volume}(R)} = \frac{|M_j(t)|}{|M(t)|-1}  - \frac{{\rm volume}(P_j({\bm v}))}{{\rm volume}(R)} = \mathcal{O}\left(\frac{1}{t}\right).
\end{align}
\end{proposition}

The proof of \cite[Prop.\ 5.4]{BlehHommJiRoedShen2012} uses the estimate that for any convex $\Omega \subset \mathbb{R}^2$ we have
\begin{align*}
{\rm area}(t \Omega) - |t \Omega \cap \mathbb{Z}^2| = \mathcal{O}(t).
\end{align*}
This estimate adapts to work for a convex region $\Omega
\subset \mathbb{R}^d$, with the area becoming volume and the error becoming $\mathcal{O}(t^{d-1})$.  This is
the only change needed to adapt the proof of \cite[Prop.\  5.4]{BlehHommJiRoedShen2012} to the present setting. 

\begin{proposition}[Prop.\ 5.3 from \cite{BlehHommJiRoedShen2012}]\label{PROP:BLEHER_5_3}
The function $P$ is Lipschitz continuous with respect to the infinity norm on $(\mathbb{R}^d)^J$ and the metric 
\begin{align*}
d(P^{(1)},P^{(2)}) = \sum_{j=1}^J {\rm vol}(P^{(1)}_j \Delta P^{(2)}_j)
\end{align*}
on $J$-tuples of subsets of $R$,
where $\Delta$ denotes the symmetric difference of sets.
\end{proposition}
The only change needed to adapt the proof of Prop.\ 5.3 from
\cite{BlehHommJiRoedShen2012} from dimension two to dimension $d \geq 2$ is
that one needs to remark that the polynomial
$Q(x)=\left(1+\frac{x}{a}\right)^d-1$ which appears in that proof is Lipschitz on the interval $[0,1]$.

\subsection{Matrix representation of multiplication by elements of $\Phi$ and their logarithms.}\label{subsec.matrix}

Multiplication by any non-zero $a \in \Phi$ corresponds to an invertible linear mapping from $\Phi$ to itself.  

\begin{lemma}\label{LEM:COMMUTING_MATRICES}
For any two non-zero $a,b \in \Phi$ the $(d+1) \times (d+1)$ dimensional matrices $A$ and $B$ representing multiplication
 by $a$ and $b$ in terms of the basis $1,\omega_1,\ldots,\omega_d$ commute.
\end{lemma}

\begin{proof}
This is an immediate consequence of commutativity $ab = ba$ in $\Phi$.
\end{proof}


Combined with Lemma \ref{LEM:COMMUTING_MATRICES} the following lemma plays a crucial role in our proof.

\begin{lemma}\label{LEM:COMMUTING_LOGS}
For any $k \geq 1$ let $A_1,\ldots,A_k$ be a commuting collection of invertible $n \times n$ matrices.
Then, there exist $n \times n$ matrices 
$L_{1},\ldots,L_{k}$ such that
\begin{align}\label{EQN:EXPONENTIAL_OF_MATRIX}
A_j = {\rm e}^{L_j}
\end{align}
for each $1 \leq j \leq k$ and such that all $k$ matrices $L_{1},\ldots,L_{k}$ commute.
\end{lemma}
As usual, the matrix exponential in (\ref{EQN:EXPONENTIAL_OF_MATRIX}) is interpreted using the standard power series for ${\rm e}^z$. 
The matrices $L_1,\ldots,L_k$ are called {\em logarithms} of $A_1,\ldots,A_k$; see, for example \cite[Sec. \ 2.3]{Hall2003}.
Although it seems that this lemma should be well-known we could not find a suitable reference.  
We thank Evgeny Mukhin and Vitaly Tarasov for providing us with the following proof.

\begin{proof}
We claim that there is an invertible $n \times n$ matrix $P$ such that for every $1 \leq j \leq k$ we have
$A_j = P D_j P^{-1}$  with each $D_j$ a block diagonal matrix
\begin{align*}
D_j = {\rm diag}(B_{j,1},\ldots,B_{j,m}),
\end{align*}
with the size of the blocks independent of $1 \leq j \leq k$, and each block having the form
\begin{align*}
B_{j,\ell} = \lambda_{j,\ell} {\rm I} + N_{j,\ell},
\end{align*}
where $\lambda_{j,\ell} \in \mathbb{C} \setminus \{0\}$ and $N_{j,\ell}$ is a nilpotent matrix for each $1 \leq j \leq k$ and $1 \leq \ell \leq m$.  Here, ${\rm I}$ denotes the identity matrix
of the appropriate dimension.

Let us first see how this claim yields the desired result.  Notice that it suffices to find commuting logarithms of the block 
diagonal matrices $D_j$, for $1 \leq j \leq k$, because the desired $L_1,\ldots,L_k$ will then be obtained by conjugating by $P$.
Moreover, matrix exponentials respect block-diagonal structure, so it suffices to find mutually commuting matrix logarithms for $k$
commuting matrices of the form
\begin{align*}
M_j = \lambda_j  {\rm I} + N_{j},
\end{align*}
where $N_j$ is nilpotent for each $1 \leq j \leq k$.  One can do this using the Mercator series to define
\begin{align*}
{\mathcal L}_j := \log \lambda_j {\rm I} - \sum_{m=1}^\infty \frac{1}{m} \left(\frac{-N_{j}}{\lambda_j}\right)^m,
\end{align*}
for each $1 \leq j \leq k$.  (One can choose any complex logarithm $\log \lambda_j$ that one likes.)
The series converge because they terminate in finitely many steps, since each $N_{j}$ is nilpotent.  Moreover ${\mathcal L}_1\ldots,{\mathcal L}_k$ commute because the $N_1,\ldots,N_k$ commute.

Now we are left to establish the claim from the beginning of the proof. Existence of the matrix $P$ that simultaneously conjugates the $A_1,\ldots,A_k$
to the desired block diagonal form is a generalization of the well-known fact
that commuting diagonalizable matrices are simultaneously diagonalizable.  We
sketch it here in the case $k=2$, leaving the straightforward generalization to
larger $k$ to the reader.

By the Jordan decomposition we can write $\mathbb{R}^n$ as a direct sum of
generalized eigenspaces of $A_1$.  Let $V$ be any one of the
generalized eigenspaces of $A_1$, corresponding to eigenvalue $\lambda$.  By definition, it consists of the vectors
in $\mathbb{R}^n$ in the kernel of $(A_1 - \lambda {\rm I})^\ell$ for some $\ell \geq 1$.
On $V$ the matrix $A_1$ is conjugate to $\lambda {\rm I} + N$ for some nilpotent matrix $N$, by the Jordan form.

Commutativity of $A_1$ and $A_2$ implies that $A_2(V)$ is a subspace of $V$.
We can therefore decompose $V$ into generalized eigenspaces of $A_2$ and, for any such
subspace $W$ of $V$, commutativity of $A_1$ and $A_2$ implies $A_1 W$ is a
subspace of $W$.  In particular $N(W)$ is a subspace of $W$.

This proves that $\mathbb{R}^n$ can be decomposed into a direct sum of spaces that are simultaneously generalized eigenspaces
of $A_1$ and $A_2$, on each of which $A_1$ is conjugate to $\lambda {\rm I} + N_1$ and $A_2$ is conjugate to $\mu {\rm I} + N_2$
for some $\lambda, \mu \in \mathbb{C} \setminus \{0\}$ and nilpotent matrices $N_1$ and $N_2$.
\end{proof}

\subsection{Reduction to analysis of normalized expansions of $u_1(t)$} \label{subsec.glue}

Here we reduce the proof of Part (2) of Theorem A' to the proof of the following theorem.

\begin{theorem}\label{THM:QUASIPERIODICITY_V2}
Under the hypotheses of Theorem A'
there is an integer $0 \leq k \leq d+1$, a Lipschitz continuous function
\begin{align*}
g_3: \mathbb{T}^k \times [0,1]^r \rightarrow \mathbb{R}^{d+1},
\end{align*}
angles ${\bm \theta} = (\theta_1,\ldots,\theta_k) \in \mathbb{T}^k$, ``rates''
${\bm \beta} = (\beta_1,\ldots,\beta_r) \in \mathbb{R}^r$, and $0 < \alpha <
1$ such that
\begin{align*}
\frac{{\bm n}(u_1(t))}{t} = g_3\left({\bm \theta} \log t, \{ {\bm \beta} \log t \} \right) + \mathcal{O}(\alpha^{\log t}).
\end{align*}
\end{theorem}

\begin{proof}[Proof of Part (2) of Theorem A' supposing Theorem \ref{THM:QUASIPERIODICITY_V2}]
Let us first summarize what Propositions \ref{PROP:BLEHER_5_1}-\ref{PROP:BLEHER_5_3} achieve.  Let
\begin{align*}
g_1: (\mathbb{R}^d)^J \rightarrow \mathbb{R}^J
\end{align*}
be given by
\begin{align*}
g_1({\bm v}) = \left(\frac{{\rm volume}(P_1({\bm v}))}{{\rm volume}(R)},\ldots,\frac{{\rm volume}(P_J({\bm v}))}{{\rm volume}(R)}  \right).
\end{align*}
Then, $g_1$ is Lipschitz continuous and satisfies that for any $t \ge 1$
\begin{align*}
g_1({\bm v}(t)) - (p_1(t),\ldots,p_J(t)) = \mathcal{O}(t^{-1}),
\end{align*}
where ${\bm v}(t)$ is given by (\ref{EQN:DEF_V_T}).

We now claim that there is a is a linear function
\begin{align*}
g_2: \mathbb{R}^{d+1} \rightarrow (\mathbb{R}^d)^J
\end{align*}
such that
\begin{align*}
{\bm v}(t) = g_2\left(\frac{{\bm n}(u_1(t))}{t} \right).
\end{align*}
For any $1 \leq j \leq J$ let $S_j$ denote the matrix expressing multiplication by $s_j \in \S$ in terms of the basis 
$1,\omega_1,\ldots,\omega_d$.  For any $1 \leq j \leq J$ we let the $j$-th component ${\bm v}_j(t)$ of ${\bm v}(t)$ be the projection of 
\begin{align*}
\frac{{\bm n}(s_j u_1(t))}{t} = S_j \frac{{\bm n}(u_1(t))}{t} \in \mathbb{R}^{d+1}
\end{align*}
onto its last $d$ components.

Therefore, given the function $g_3: \mathbb{T}^k \times [0,1]^r \rightarrow \mathbb{R}^{d+1}$
whose existence is asserted by Theorem \ref{THM:QUASIPERIODICITY_V2}, we can let $g = g_1 \circ g_2 \circ g_3$ so that
\begin{align*}
g\left({\bm \theta} \log t, \{ {\bm \beta} \log t \} \right) &= g_1\left(g_2\left(\frac{{\bm n}(u_1(t))}{t} + \mathcal{O}(\alpha^{\log t})
\right)\right) \\ &= g_1\left({\bm v}(t) + \mathcal{O}(\alpha^{\log t})\right)\\  &= (p_1(t),\ldots,p_J(t)) + \mathcal{O}(t^{-1}) + \mathcal{O}(\alpha^{\log t}),
\end{align*}
with the last equality using that $g_1$ is Lipschitz.  Since $t^{-1} = (1/{\rm e})^{\log t}$ this proves the claim.
\end{proof}

\subsection{Analysis of normalized expansions of $u_1(t)$ as $t \rightarrow \infty$.}\label{subsec.dynamics}
We will now use results from linear algebra to finish the proof of Theorem \ref{THM:QUASIPERIODICITY_V2}. We begin by deriving an explicit formula for $u_1(t) = u(t)^{-1}$.
Let $r = r_1+r_2-1$ denote the rank of the unit group $\mathbb{Z}_\Phi^\times$ and let $\epsilon_1,\ldots,\epsilon_r$ be a basis for a finite index subgroup of the multiplicative group
$\mathbb{Z}_\Phi^\times$.  Recall that the image of $\mathbb{Z}_\Phi^\times$ under the mapping $\varphi$ given in (\ref{EQN:DEF_VARPHI}) forms a lattice in the hyperplane
$H \subset \mathbb{R}^{r_1+r_2}$ defined by (\ref{eqn.Hyper}).  
Therefore we can use the coordinate system
\begin{align*}
(y_1,\ldots,y_r) \mapsto y_1 \varphi(\epsilon_1) + \cdots + y_r \varphi(\epsilon_r)
\end{align*}
on $H$.
In these coordinates the image under $\phi$ of the group generated by $\epsilon_1,\ldots,\epsilon_r$ becomes the integer lattice $\mathbb{Z}^r$.
The path ${\bm x}(t)$ defined in (\ref{EQN:X_OF_T}) becomes
\begin{align}\label{DEF_VECTOR_W}
{\bm x}(t) =  -{\bm w} \log t
\end{align}
for some suitable non-zero vector ${\bm w} = (w_1,\ldots,w_r)$.
We can then use
\begin{align*}
u_1(t) = u(t)^{-1}=\epsilon_1^{\lfloor w_1 \log t\rfloor} \cdots \epsilon_r^{\lfloor w_r \log t\rfloor}
\end{align*}
as the unit in Part (1) of Theorem A'.

Let $U(t)$ denote the matrix representing multiplication by $u_1(t)$ in the basis $1,\omega_1,\ldots,\omega_d$.  We have
\begin{align*}
U(t) = E_1^{\lfloor w_1 \log t\rfloor} \cdots E_r^{\lfloor w_r \log t\rfloor},
\end{align*}
where $E_1,\ldots,E_r$ are the matrices representing multiplication by the units $\epsilon_1,\ldots,\epsilon_r$.

Let us approximate $U(t)$ by a continuous version.
According to Lemmas~\ref{LEM:COMMUTING_MATRICES} and \ref{LEM:COMMUTING_LOGS} we can choose logarithms $L_1,\ldots,L_r$ of the matrices $E_1,\ldots,E_r$ in a way that they all commute.  Let
\begin{align}\label{EQN:DEF_U_TILDE}
\tilde{U}(t) = {\rm e}^{w_1 (\log t) L_1  + \cdots + w_r (\log t) L_r} = {\rm e}^{(\log t) L},
\end{align}
where 
\begin{align*}
L := w_1 L_1  + \cdots +w_r L_r.
\end{align*}
We remark that, since $L_1,\ldots , L_r$ commute, we also have that
\begin{align}\label{EQN:DEF_U_TILDE_V2}
\tilde{U}(t) = E_1^{w_1 \log t} \cdots E_r^{w_r \log t},
\end{align}
since the definition of the real power of a matrix gives $E_j^{w_1 \log t} := {\rm e}^{w_1 \log t L_j}$ for $1 \leq j \leq r$.  
Note that making a different choice of matrix logarithm can lead to a different value of the real power of a matrix, just like for the real power of a real number.  However, we have fixed our choices of logarithms once and for all when we defined (\ref{EQN:DEF_U_TILDE}).

For any $t \ge 1$, again using the commutativity of $L_1,\ldots , L_r$, we have that
\begin{align}\label{EQN:FACTORIZATION}
U(t) = A(t) \tilde{U}(t) 
\end{align}
where 
\begin{align}\label{EQN_A_OF_T}
A(t) = {\rm e}^{-\{w_1 \log t\} L_1   \cdots -\{w_r \log t\} L_r}.
\end{align}
We think of $A(t)$ as the ``multiplicative error'' between $U(t)$ and our continuous approximation $\tilde{U}(t)$.  It ranges
over a compact subset of the space of invertible $(d+1) \times (d+1)$ matrices.


\begin{remark}
The product $A(t) \tilde{U}(t)$ is a real matrix, because
$U(t)$ is.  However, the matrices $A(t)$ an $\tilde{U}(t)$ are not (necessarily) real because the logarithms $L_1,\ldots,L_r$
are not necessarily real.
\end{remark}

To prove Theorem \ref{THM:QUASIPERIODICITY_V2} we must estimate
\begin{align*}
\frac{{\bm n}(u_1(t))}{t} = \frac{1}{t} U(t) \ {\bm e}_1 = A(t) \, \left(\frac{1}{t} \tilde{U}(t) \ {\bm e}_1\right),
\end{align*}
where ${\bm e}_1 = (1,0,\ldots,0)^T$.  
To do this we write
\begin{align*}
\frac{1}{t} \tilde{U}(t) = {\rm e}^{-(\log t) {\rm I}} {\rm e}^{(\log t) L} =  {\rm e}^{(\log t) (L - {\rm I})},
\end{align*}
where ${\rm I}$ denotes the $(d+1) \times (d+1)$ identity matrix.

\begin{lemma}\label{LEM:EIGENVALUES}
The real part of every eigenvalue of $L-{\rm I}$ is non-positive and there exist eigenvalues whose real part is $0$.
In the Jordan canonical form for $L-{\rm I}$, each of the purely imaginary eigenvalues corresponds to 
a $1 \times 1$ (i.e.\ trivial) Jordan block.\end{lemma}

\begin{proof}
If $a,b: [1,\infty) \rightarrow \mathbb{R}$ are functions, we will 
use the asymptotic notation $a(t) \asymp b(t)$ to denote that there exist
constants $C_1, C_2 > 0$ such that for every $t \geq 1$ we have 
\begin{align*}
C_1 \leq \frac{a(t)}{b(t)} \leq C_2.
\end{align*}

Recall from Section \ref{sec.Prelim} that, because ${\bm \omega}$ is badly approximable, the transference principle implies that
there is a $K' > 0$ such that each spacing satisfies 
\begin{align*}
\Delta_i(t) \leq \frac{K'}{t^d}.
\end{align*}
Focusing on the smallest spacing, Part (1) of Theorem A' implies that $\Delta_1(t) = s_j(t) u_1(t)$ for some $s_j(t)$ in the finite set $\S$.  
This gives
\begin{align*}
\Delta_1(t) = (1,\omega_1,\ldots,\omega_d) \ S_j(t) \ U(t) \ {\bm e}_1  \leq \frac{K'}{t^d},
\end{align*}
where $S_j(t)$ is the matrix representing multiplication by $s_j(t)$ in the basis $\{1,\omega_1,\ldots,\omega_d\}$.
Since the vector $(\omega_1,\ldots,\omega_d)$ is badly approximable, we find that there is a constant $C_1 > 0$ such that
\begin{align*}
|S_j(t) \ U(t) \ {\bm e}_1| > C_1 t
\end{align*}
for every $t \ge 1$.  

Also observe that, since
\begin{align*}
\Delta_1(t) = \{({\bm m}_1 - {\bm m}_2) \cdot {\bm \omega}\}
\end{align*}
for some ${\bm m}_1, {\bm m}_2 \in M(t)$, there is a constant $C_2 > 0$ such that for every $t \ge 1$ we have
\begin{align*}
|S_j(t) \ U(t) \ {\bm e}_1| < C_2 t.
\end{align*}

In summary, we have
\begin{align*}
|S_j(t) \ U(t) \ {\bm e}_1| \asymp t.
\end{align*}

Since the $S_j(t)$ range over a finite set of invertible matrices (corresponding to multiplication by elements of the finite set $S$) we
conclude that
\begin{align}\label{EQN:DESIRED_ASYMPTOTICS_E1}
|U(t) \ {\bm e}_1| \asymp t.
\end{align}
For each $1 \leq i \leq d$, let $W_i$ denote multiplication by $\omega_i$ in the basis $\{1,\omega_1,\ldots,\omega_d\}$.  We then have
\begin{align*}
|U(t) \ {\bm e}_{i+1}| = |U(t) \ W_i \ {\bm e}_{1}| = |W_i \ U(t) {\bm e}_{1}| \asymp t.
\end{align*}
The last equality holds because of Lemma \ref{LEM:COMMUTING_MATRICES}, and the last assertion 
follows from (\ref{EQN:DESIRED_ASYMPTOTICS_E1}) because $W_i$ is invertible.  

Finally, since $U(t) = A(t) \tilde{U}(t)$ with $A(t)$ given by (\ref{EQN_A_OF_T}) and hence varying over a compact set of invertible matrices,
we have 
\begin{align*}
|\tilde{U}(t) \ {\bm e}_i| \asymp t.
\end{align*}
for each $1 \leq i \leq d+1$.  In other words, this gives that for each $1 \leq i \leq d+1$ we have
\begin{align}\label{ESTIMATE_1}
\left|{\rm e}^{\log t (L - {\rm I})}  {\bm e}_i\right| = \left|\frac{1}{t} \tilde{U}(t) {\bm e}_i\right| \asymp 1.
\end{align}

Now we will use the Jordan decomposition $L - {\rm I} = P J P^{-1}$, where $P$ is an invertible matrix and $J$ is in Jordan form.
It follows from the power series definition of the matrix exponential that
\begin{align*}
{\rm e}^{(\log t) (L - {\rm I})} = P {\rm e}^{(\log t) J} P^{-1}.
\end{align*}
Since $P$ is invertible, (\ref{ESTIMATE_1}) implies for each $1 \leq i \leq d+1$ that we have
\begin{align}\label{ESTIMATE_2}
\left|{\rm e}^{(\log t) J} P^{-1}  {\bm e}_i \right| \asymp 1.
\end{align}

Let us consider the upper left Jordan block, which we suppose is $k \times k$:
\begin{align*}
J_1 = \left[\begin{array}{ccccccc} 
\lambda	 & 1	 & 0	 & 0	& \cdots	 & 0	& 0 	\\ 
 0	 & \lambda & 1	 & 0	& \cdots 	 & 0	& 0     \\ 
\vdots   &  \vdots & \vdots & \vdots & & \vdots & \vdots\\
0       & 0      & 0     & 0    & \cdots         & \lambda & 1 \\
0  	& 0 	 & 0 	 & 0 	& \cdots 	 & 0 &    \lambda 
\end{array} \right] = \lambda {\rm I} + N,
\end{align*}
where $N$ is the nilpotent matrix whose only non-zero entries are ones on the ``super-diagonal'' of $J_1$.
It satisfies $N^k = 0$.
Since ${\rm I}$ commutes with every matrix we have
\begin{align*}
{\rm e}^{(\log t) J_1} = {\rm e}^{(\log t \lambda) {\rm I}} \ {\rm e}^{(\log t) N}  = {\rm diag}(e^{\log t \lambda},\ldots,e^{\log t \lambda})
\ 
 p((\log t) N)
\end{align*}
where 
\begin{align*}
p(x) = 1 + x + \frac{x^2}{2!} + \cdots + \frac{x^{k-1}}{(k-1)!}.
\end{align*}
From this, one can check that the $(1,k)$ entry of ${\rm e}^{(\log t) J_1}$ equals
\begin{align*}
\frac{1}{(k-1)!}(\log t)^{k-1} {\rm e}^{\log t \lambda}
\end{align*}
 and that all other entries
have moduli that are smaller, asymptotically as $t\rar\infty$, by at least a factor of $\log t$.  

Because ${\bm e}_1,\ldots, {\bm e}_{d+1}$ form a basis for
$\mathbb{R}^{d+1}$ there exists $1 \leq j \leq d+1$ such that the $k$-th entry
of $P^{-1} {\bm e}_j$ is non-zero.  It follows from the previous paragraph that 
\begin{align*}
|{\bm e}_1\cdot({\rm e}^{(\log t) J} P^{-1}  {\bm e}_j)| \asymp (\log t)^{k-1} e^{\log t \ {\rm Re}(\lambda)}.
\end{align*}
Combined with (\ref{ESTIMATE_2}) this implies that 
${\rm Re}(\lambda) \leq 0$ and that if ${\rm Re}(\lambda) = 0$ then $k=1$.

By permuting the Jordan blocks of $J$ we find that the same holds for every other Jordan block.  Finally, existence of at least one purely imaginary eigenvalue is needed for the lower bound implied by (\ref{ESTIMATE_2}) to hold.
\end{proof}

\begin{proof}[Proof of Theorem \ref{THM:QUASIPERIODICITY_V2}]
By Lemma \ref{LEM:EIGENVALUES}, we can write $L - {\rm I} = P J P^{-1}$ with $J$ being block-diagonal of the form
\begin{align*}
J = {\rm diag}(2\pi i \theta_1,\ldots , 2\pi i \theta_k,J_{k+1},\ldots,J_\ell)
\end{align*}
with $\theta_1,\ldots,\theta_k \in \mathbb{R}$ for some $1 \leq k \leq d+1$ and 
with the blocks $J_{k+1},\ldots ,J_\ell$ all corresponding to eigenvalues with real parts less than some $\gamma < 0$.
Note that at this step we have selected the angles ${\bm \theta} = (\theta_1,\ldots,\theta_k)$ which are asserted to exist in the statement of the theorem.

Consider the $(d+1) \times (d+1)$ diagonal matrix:
\begin{align*}
\hat{J} := {\rm diag}(2\pi i \theta_1,\ldots ,2\pi i \theta_k,0,\ldots,0).
\end{align*}
Let $e^{\gamma} < \alpha < 1$.  Then, 
it follows from the calculations of exponentials of Jordan blocks in the end of the proof of Lemma \ref{LEM:EIGENVALUES} that for any $t \geq 1$ we have
\begin{align}\label{ESTIMATE}
{\rm e}^{(\log t) J} - {\rm e}^{(\log t) \hat{J}} = \mathcal{O}(\alpha^{\log t}).
\end{align}
Here we mean that the modulus of each corresponding component of the difference is $\mathcal{O}(\alpha^{\log t}).$

Define $g_3: \mathbb{T}^k \times [0,1]^r \rightarrow \mathbb{R}^{d+1}$ by
\begin{align}\label{eqn.g3def}
g_3({\bm \psi},{\bm x}) :=  
{\rm Re}\left({\rm e}^{-x_1 L_1   \cdots -x_r L_r}  P \ {\rm diag}({\rm e}^{2\pi i \psi_1},\ldots,{\rm e}^{2\pi i\psi_k},0,\ldots,0) \  P^{-1}  {\bm e}_1\right).
\end{align}
Clarifications:
\begin{enumerate}
\item Here, as usual, we denote the angles ${\bm \psi} \in \mathbb{T}^k$ by their lifts in $\mathbb{R}^k$.  However the formula
clearly only depends on the angles themselves.
\item The ${\rm Re}$ denotes that we are taking the real part of each component of the resulting vector.
\item This function is differentiable, hence Lipschitz.
\end{enumerate}

If we define our rates by ${\bm \beta} = {\bm w}$  (see \eqref{DEF_VECTOR_W}) then we have 
\begin{align*}
g_3({\bm \theta} \log t, \{\bm \beta \log t\}) = {\rm Re}\left(A(t) P {\rm e}^{\log t \hat{J}} P^{-1} {\bm e}_1\right).
\end{align*}
Finally, observe that
\begin{align*}
\frac{{\bm n}(u_1(t))}{t} = A(t) \, \left(\frac{1}{t} \tilde{U}(t) \ {\bm e}_1\right) = A(t) {\rm e}^{\log t (L - {\rm I})}{\bm e}_1 = A(t) P {\rm e}^{\log t J} P^{-1} {\bm e}_1.
\end{align*}
Since $A(t)$ ranges over a compact set of matrices, the result follows from (\ref{ESTIMATE}) and the fact that
$\frac{{\bm n}(u_1(t))}{t}$ is real.

\end{proof}


\section{Quasiperiodicity of finite words and Proof of Theorem B'}
\label{SEC:QP_ON_WORDS_AND_PROOF_THM_BPRIME}

We will show that a relatively simple modification of the proof of Theorem A' yields the
following stronger statement.  

\begin{THMC}
Suppose that $1,\omega_1,\ldots
,\omega_d\in\RR$ form a $\QQ$-basis for an algebraic number field $\Phi$
and let $r \geq 1$ denote the rank of $\ZZ_\Phi^\times$.
Let $\S$ be the finite set given by Part (1) of Theorem A'.

For each choice of $1 \leq j_0,\ldots,j_l  \leq J$ let $p_{j_0,\ldots,j_l}(t)$ denote the proportion of the points
from $\{y_i(t) \, : \, 1 \leq i \leq \ell(t) - l-1\}$ such that 
\begin{align}\label{EQN_WORD_CONDITION}
\delta_{i}(t) &= y_{i+1}(t) - y_{i}(t) = s_{j_0} u(t)^{-1}, \nonumber \\
\delta_{i+1}(t) &= y_{i+2}(t) - y_{i+1}(t) = s_{j_1} u(t)^{-1}, \nonumber  \\
& \vdots  \\
\delta_{i+l-1}(t) &= y_{i+l}(t) - y_{i+l-1}(t) = s_{j_{l-1}}  u(t)^{-1}, \quad \mbox{and} \nonumber \\
\delta_{i+l}(t) &= y_{i+l+1}(t) - y_{i+l}(t) = s_{j_l} u(t)^{-1}. \nonumber
\end{align}
Equivalently, $p_{j_0,\ldots,j_l}(t)$ is the proportion of the points $y_i(t)$ such
that the sequence of $l+1$ consecutive rescaled spacings starting at $y_i(t)$ forms the word $s_{j_0} s_{j_2} \ldots s_{j_l}$.

Then,
there is an integer $0 \leq k \leq d+1$, a Lipschitz continuous function
\begin{align*}
h: \mathbb{T}^k \times [0,1]^r \rightarrow P := \left\{(p_{1,\ldots,1},\ldots,p_{J,\ldots,J}) \, : \sum p_{j_0,\ldots,j_l} = 1, p_{j_0,\ldots,j_l} \geq 0\right\},
\end{align*}
angles ${\bm \theta} = (\theta_1,\ldots,\theta_k) \in \mathbb{T}^k$, rates ${\bm \beta} = (\beta_1,\ldots,\beta_r) \in \mathbb{R}^r$, and $0 < \alpha < 1$ such that
\begin{align*}
(p_{1,\ldots,1,1}(t),p_{1,\ldots,1,2}(t),\ldots,p_{J,\ldots,J,J}(t)) = h\left({\bm \theta} \log t, \{ {\bm \beta} \log t \} \right) +
\mathcal{O}(\alpha^{\log t}).
\end{align*}
\end{THMC}

\begin{proof}
The only changes to the proof of Theorem A' that are necessary are adaptations to Section~\ref{subsec.Partitions}
about the partitions.

For any length $l+1$ word $j_0 j_1 \ldots j_l \in J^{l+1}$ let $Y_{j_0 j_1 \ldots j_l}(t)$ be the subset of those
$$\{y_i(t) \, : \, 1 \leq i \leq \ell(t) - l-1\}$$ such that (\ref{EQN_WORD_CONDITION}) holds.
Let $M_{j_0 j_1 \ldots j_l}(t)$ be the set of vectors ${\bm m} \in M(t)$ such that $\{{\bm m} \cdot {\bm \omega}\} \in Y_{j_0 j_1 \ldots j_l}(t)$.  These sets form a partition of $M(t)$ up to the $l$ points corresponding
to $\{y_{\ell(t)-l},\ldots,y_{\ell(t)}\}$.
As in the proof of Theorem A' we have for any word $j_0 j_1 \ldots j_l$ that
\begin{align*}
p_{j_0 j_1 \ldots j_l}(t) = \frac{|M_{j_0 j_1 \ldots j_l}(t)|}{|M(t)| - l}.
\end{align*}

Let us first consider the case of words of length two.
As in Proposition \ref{PROP:BLEHER_5_1}, for any $1 \leq j \leq J$ let
\begin{align*}
{\bm v}_j(t) = {\bm m}(s_j u_1(t)).
\end{align*}
Then, we claim that
\begin{align*}
M_{j_0 j_1}(t) = M_{j_0}(t) \cap \left(M_{j_1}(t) - {\bm v}_{j_0}(t)\right).
\end{align*}
To see this, notice that if ${\bm m}(y_i(t)) = {\bm v}$ then to have
$\delta_i(t) = s_{j_0} u(t)^{-1}$ we need to have ${\bm v} \in M_{j_0}(t)$ and in order
to have $\delta_{i+1}(t) = s_{j_1} u(t)^{-1}$ we need to have  
${\bm v} + {\bm v}_{j_0}(t) \in M_{j_1}(t)$.  Here, we are using that since $\delta_i(t) = s_{j_0} u(t)^{-1}$
we have ${\bm m}(y_{i+1}(t)) = {\bm v} +  {\bm v}_{j_0}(t)$.

A simple induction yields the following formula in the general case:
\begin{align*}
M_{j_0 j_1 \ldots j_l}(t) = M_{j_0}(t) \cap \left(M_{j_1}(t) - {\bm v}_{j_0}(t)\right) \cap \cdots \cap \left(M_{j_l}(t) - {\bm v}_{j_0}(t)-{\bm v}_{j_1}(t)-\cdots-{\bm v}_{j_{l-1}}(t)\right).
\end{align*}
In each of these formulae the sets $M_j(t)$ are defined as in Section \ref{subsec.Partitions} and given by (\ref{EQN:PARTITIONS}).

Denote the power set of our region $R$ by $\mathcal{P}(R)$ and let
\begin{align*}
\mathscr{P} : (\mathbb{R}^d)^J \rightarrow \mathcal{P}(R)^{J^l}
\end{align*}
be the mapping which sends ${\bm v} = ({\bm
v_1},\ldots,{\bm v_J}) \in (\mathbb{R}^d)^J$ to $(\mathscr{P}_{11\ldots1}({\bm
v}),\ldots,\mathscr{P}_{JJ\ldots J}({\bm v})) \in \mathcal{P}(R)^{J^l}$, where,
\begin{align*}
\mathscr{P}_{j_0 j_1 \ldots j_l}({\bm v}) = P_{j_0}({\bm v}) \cap \left(P_{j_1}({\bm v}) - {\bm v}_{j_0}\right) \cap \cdots \cap \left(P_{j_l}({\bm v}) - {\bm v}_{j_0}-{\bm v}_{j_1}-\cdots-{\bm v}_{j_{l-1}}\right).
\end{align*}
Here, 
\begin{align*}
P_j({\bm v}) = [R \cap (R - {\bm v_j})] \setminus \bigcup_{i=1}^{j-1}(R - {\bm v_i}),
\end{align*}
is the same formula as from Section \ref{subsec.Partitions}.

We then claim that the appropriately generalized versions of Propositions \ref{PROP:BLEHER_5_2} and \ref{PROP:BLEHER_5_3}, expressed in terms of the definitions and formulae above for $M_{j_0 j_1 \ldots j_l}(t)$ and $\mathscr{P}_{j_0 j_1 \ldots j_l}({\bm v})$, also hold in this context.  We leave this to the reader to check.  The remainder  of the proof holds exactly as the proof of Theorem A'
\end{proof}

\noindent
Theorem B' is now an immediate consequence of Theorem C in the case of words
of length two.
\qed (Theorem B')


\section{Worked example: a totally real cubic field}\label{sec.example}

In this section we will work out the details of our quasiperiodicity theorem (Part (2) of Theorem A') in a particular example which highlights many of the important steps in its proof. For our example we take a totally real cubic field of smallest discriminant, for which we can use the table in \cite{CusiScho1987} to identify a pair of generators for a finite index subgroup of the group of units in $\ZZ_\Phi^\times$. We also take the region $R$ to be the half open unit square, so that $R(t)=[0,t)^2$ and $M(t)=\ZZ^2\cap [0,t)^2$.

Let $\omega_1$ be the smallest real root of the cubic polynomial $f(x)=x^3-7x^2+14x-7$, and let $\omega_2=\omega_1^2$. Then $1,\omega_1,$ and $\omega_2$ form a $\QQ$-basis for the algebraic number field $\Phi=\QQ(\omega_1)$ of degree $d+1=3$ over $\QQ$, and the ring of integers of $\Phi$ is $\ZZ_\Phi=\ZZ[\omega_1]$ (see \cite{CusiScho1987}).

All three of the roots of $f(x)$ are real and positive, so let us list them as $0<\alpha_1<\alpha_2<\alpha_3$ (note that $\omega_1=\alpha_1$). For each $1\le i\le 3$ let $\sigma_i$ be the embedding of $\Phi$ into $\RR$ which maps $\omega_1$ to $\alpha_i$. It is clear that $r_1=3$ and $2r_2=0$, so the rank of the group of units is $r_1+r_2-1=2$. Now we let $\epsilon_1,\epsilon_2\in\Phi$ be defined by
\begin{equation}
\epsilon_1=2-4\alpha_1+\alpha_1^2\quad\text{and}\quad \epsilon_2=-5+5\alpha_1-\alpha_1^2.
\end{equation}
It follows from \cite{CusiScho1987} that $\epsilon_1$ and $\epsilon_2$ generate a finite index subgroup of the group of units of $\ZZ_\Phi$.

Next, following \eqref{DEF_VECTOR_W}, we let $\bm{\beta}=(\beta_1,\beta_2)$ be determined by
\begin{equation*}
\begin{pmatrix}
\beta_1\\ \beta_2
\end{pmatrix}
= \begin{pmatrix}
\log |\epsilon_1| & \log |\epsilon_2|\\ \log |\sigma_2(\epsilon_1)| & \log |\sigma_2(\epsilon_2)|
\end{pmatrix}^{-1}
\begin{pmatrix}
-2\\1
\end{pmatrix},
\end{equation*}
so that 
\begin{equation*}
\bm{\beta}\approx (1.96080, -0.70061),
\end{equation*}
and for $t\ge 1$ we set
\begin{equation*}
u_1(t)=\epsilon_1^{\lfloor\beta_1\log t\rfloor}\epsilon_2^{\lfloor\beta_2\log t\rfloor}.
\end{equation*}
Note for later that the fact the $\epsilon_1$ and $\epsilon_2$ are units also gives that
\begin{equation*}
\beta_1\log |\sigma_3(\epsilon_1)|+\beta_2\log|\sigma_3(\epsilon_2)|=-\beta_1\log |\epsilon_1\sigma_2(\epsilon_1)|-\beta_2\log|\epsilon_2\sigma_2(\epsilon_2)|=1.
\end{equation*}

The first claim of Theorem A' is that there is a finite set $S\subseteq\Phi$ with the property that, for any $t\ge 1$, every spacing $\Delta_i(t)$ has the form $u_1(t)s$ for some $s\in \S$. In order to identify such a set, we must first derive an upper bound for $\Delta_i(t)$. As indicated by our proof above, we will do this using Cassels's transference principle. First notice that if $(m_1,m_2)\in \ZZ^2\setminus\{\bm{0}\}$ satisfies $|m_1|,|m_2|\le t$, and if $m_0\in\ZZ$ is chosen so that
\begin{equation*}
\|\bm{m}\cdot\bm{\omega}\|=m_0+m_1\omega_1+m_2\omega_2,
\end{equation*}
then we have that
\begin{equation*}
|m_0|\le|m_1\omega_1+m_2\omega_2|+1/2\le (1/2+\alpha_1+\alpha_1^2)t.
\end{equation*}
This in turn gives for $j=2$ and 3 that
\begin{equation*}
|m_0+m_1\alpha_j+m_2\alpha_j^2|\le (1/2+\alpha_1+\alpha_1^2+\alpha_j+\alpha_j^2)t,
\end{equation*}
and it follows that
\begin{align*}
\|\bm{m}\cdot\bm{\omega}\|=\frac{\mathrm{Norm}(m_0+m_1\alpha_1+m_2\alpha_1^2)}{|m_0+m_1\alpha_2+m_2\alpha_2^2|\cdot |m_0+m_1\alpha_3+m_2\alpha_3^2|}\ge\frac{1}{Kt^2},
\end{align*}
with 
\begin{equation*}
K=(1/2+\alpha_1+\alpha_1^2+\alpha_2+\alpha_2^2)(1/2+\alpha_1+\alpha_1^2+\alpha_3+\alpha_3^2).
\end{equation*}
The inhomogeneous transference principle \cite[Section V, Theorem VI]{Cass1957} then implies that, for any $t\ge 1$ and for any real number $\gamma$, there is an integer solution $(n_1,n_2)\in\ZZ^2$ to the inequality
\begin{equation*}
\|\bm{n}\cdot\bm{\omega}-\gamma\|\le \frac{\lfloor K\rfloor+1}{2Kt^2}
\end{equation*}
satisfying
\begin{equation*}
|n_1|,|n_2|\le \left(\frac{\lfloor K\rfloor+1}{2}\right)t.
\end{equation*}
Rescaling, we conclude that there is always an integer solution $(n_1,n_2)\in\ZZ^2$ to the inequality
\begin{equation*}
\|\bm{n}\cdot\bm{\omega}-\gamma\|\le \frac{(\lfloor K\rfloor+1)^3}{8Kt^2},
\end{equation*}
with $|n_1|,|n_2|\le t$. It follows that, for $t\ge 1$, every spacing $\Delta_i(t)$ must satisfy the inequality
\begin{equation}\label{eqn.ExSpacingBd}
\Delta_i(t)\le \frac{(\lfloor K\rfloor+1)^3}{4Kt^2}.
\end{equation}
Since every spacing $\Delta_i(t)$ must be of the form $\bm{m}\cdot\bm{\omega}$ for some $\bm{m}\in M(t)$,
from the above discussion we also have for $j=2$ and 3 that
\begin{equation*}
\sigma_j(\Delta_i(t))\le (1/2+\alpha_1+\alpha_1^2+\alpha_j+\alpha_j^2)t.
\end{equation*}
Next, with help from the math software Sage, we find that
\begin{align*}
|u_1(t)^{-1}|&=|\epsilon_1|^{-\lfloor\beta_1\log t\rfloor}|\epsilon_2|^{-\lfloor\beta_2\log t\rfloor}\\
&=t^2\exp\left(\{\beta_1\log t\}\log|\epsilon_1|+\{\beta_2\log t\}\log|\epsilon_2|\right)\\
&\le |\epsilon_2|t^2,
\end{align*}
that
\begin{align*}
|\sigma_2(u_1(t)^{-1})|&=|\sigma_2(\epsilon_1)|^{-\lfloor\beta_1\log t\rfloor}|\sigma_2(\epsilon_2)|^{-\lfloor\beta_2\log t\rfloor}\\
&=t^{-1}\exp\left(\{\beta_1\log t\}\log|\sigma_2(\epsilon_1)|+\{\beta_2\log t\}\log|\sigma_2(\epsilon_2)|\right)\\
&\le \frac{|\sigma_2(\epsilon_1)||\sigma_2(\epsilon_2)|}{t},
\end{align*}
and that
\begin{align*}
|\sigma_3(u_1(t)^{-1})|&=|\sigma_3(\epsilon_1)|^{-\lfloor\beta_1\log t\rfloor}|\sigma_3(\epsilon_2)|^{-\lfloor\beta_2\log t\rfloor}\\
&=t^{-1}\exp\left(\{\beta_1\log t\}\log|\sigma_3(\epsilon_1)|+\{\beta_2\log t\}\log|\sigma_3(\epsilon_2)|\right)\\
&\le \frac{|\sigma_3(\epsilon_1)|}{t}.
\end{align*}
This means that the Minkowski embedding $\sigma(u_1(t)^{-1}\Delta_i)$ of $u_1(t)^{-1}\Delta_i$ into $\RR^3$ is a point of the lattice $\Gamma=\sigma(\ZZ_\Phi)$ which lies in the box
\begin{equation*}
[-K_1,K_1]\times [-K_2,K_2]\times [-K_3,K_3],
\end{equation*}
with
\begin{align*}
K_1&= \frac{(\lfloor K\rfloor+1)^3|\epsilon_2|}{4K},\\ K_2&=(1/2+\alpha_1+\alpha_1^2+\alpha_2+\alpha_2^2)|\sigma_2(\epsilon_1)||\sigma_2(\epsilon_2)|,~\text{and}\\ K_3&=(1/2+\alpha_1+\alpha_1^2+\alpha_3+\alpha_3^2)|\sigma_3(\epsilon_1)|.
\end{align*}
Therefore, for our finite set $\S$ we may take the collection of all points of $\Gamma$ which lie in this box. Unfortunately, here there is a bit of a disappointment. The volume of the box defined above is approximately $11,034,177$, while a fundamental domain for $\Gamma$ has volume $7$. This means that the number of lattice points in the box is close to $10^6$. While it is not computationally infeasible to find and list all of these points, further computations of the areas of the regions from Proposition \ref{PROP:BLEHER_5_2} become unwieldy. They are also somewhat unenlightening, because most of the regions end up being empty- in all cases we have computed, which includes all $t\le 300$, there are no more than 10 distinct spacings. However, we can still continue further to explore the quasiperiodic behavior of the function $g_3$ from the statement of Theorem \ref{THM:QUASIPERIODICITY_V2}.

The linear transformations $E_1$ and $E_2$ of $\Gamma$ determined by multiplication by $\epsilon_1$ and $\epsilon_2$ in $\Phi$ (with respect to the basis $\sigma(1),\sigma(\omega_1),\sigma(\omega_2)$ of $\Gamma$) are given by
\begin{equation*}
E_1=\begin{pmatrix}
2&7&21\\ -4& -12& -35\\  1&   3&   9
\end{pmatrix}\quad\text{and}\quad E_2=\begin{pmatrix} -5 &  -7 & -14\\
 5&   9&   21\\ -1&  -2&  -5
\end{pmatrix}.
\end{equation*}
These matrices commute, and they are diagonalizable, therefore they are simultaneously diagonalizable. Explicitly, let
\begin{equation*}
\lambda_1=2-\alpha_1,\quad\lambda_2=2-4\alpha_1+\alpha_1^2,\quad\lambda_3=-5+5\alpha_1-\alpha_1^2,
\end{equation*}
let
\begin{equation*}
D_1=\mathrm{diag}(\lambda_1,\lambda_2,\lambda_3),\quad D_2=\mathrm{diag}(\lambda_2,\lambda_3,\lambda_1),
\end{equation*}
and let 
\begin{equation*}
Q=\begin{pmatrix}
1&1&1\\
-3+2\alpha_1-(3/7)\alpha_1^2&-\alpha_1+(1/7)\alpha_1^2&-1-\alpha_1+(2/7)\alpha_1^2\\
1-(5/7)\alpha_1+(1/7)\alpha_1^2&(1/7)\alpha_1&(4/7)\alpha_1-(1/7)\alpha_1^2
\end{pmatrix}.
\end{equation*}
Then we have for $i=1$ and 2 that
\begin{equation*}
E_i=QD_iQ^{-1}.
\end{equation*}
Noting that $\lambda_1>0$ while $\lambda_2,\lambda_3<0$, and choosing a branch of the logarithm which includes both the positive and negative real axes, a pair of commuting logarithms of $E_1$ and $E_2$ is given by
\begin{align*}
L_1&=Q \ \mathrm{diag}(\log \lambda_1, \ \log |\lambda_2|+i\pi, \ \log |\lambda_3|+i\pi) \ Q^{-1},~\text{and}\\
L_2&=Q \ \mathrm{diag}(\log|\lambda_2|+i\pi, \ \log |\lambda_3|+i\pi, \ \log \lambda_1) \ Q^{-1}.
\end{align*}
Now, following the proof of Theorem \ref{THM:QUASIPERIODICITY_V2}, we take $L=\beta_1L_1+\beta_2L_2$. In this case, again using Sage, we have that
\begin{equation*}
L-I=PJP^{-1}
\end{equation*}
with
\begin{equation*}
P\approx \begin{pmatrix} -0.52319 &-0.48157  &0.82671 \\
 0.83239 &  0.83647 &  -0.55556 \\
 -0.18274 & -0.26156 &  0.08893 \end{pmatrix}
\end{equation*}
and
\begin{equation*}
J\approx  \mathrm{diag}(6.16003~i , -2.20103~i, -3.00000 + 3.95900~i).
\end{equation*}
These numbers have been computed using (complex) double float precision, but for readability we have rounded them to five digits. This means that $k=2$ in the statement of Theorems A' and \ref{THM:QUASIPERIODICITY_V2}, and that
\begin{equation*}
\bm{\theta}\approx\left(\frac{6.16003}{2\pi},\frac{-2.20103}{2\pi}\right).
\end{equation*}
Finally, with $g_3$ defined as in \eqref{eqn.g3def}, we have that
\begin{equation*}
\frac{\bm{n}(u_1(t))}{t}=g_3(\bm{\theta}\log t,\{\bm{\beta}\log t\})+O(\alpha^{\log t}),
\end{equation*}                                    
for any $\alpha>e^{-3}$. Below is a table comparing the actual values of $\bm{n}(u_1(t))$ with the approximate values given by  $tg_3(\bm{\theta}\log t,\{\bm{\beta}\log t\}).$ The values of $t$ have been sampled along the sequence $\lfloor 10^{i/2}\rfloor$.

\begin{equation*}
\begin{array}{|c|c|c|c|}
\hline	
i & t=\lfloor 10^{i/2}\rfloor & \bm{n}(u_1(t)) & tg_3(\bm{\theta}\log t,\{\bm{\beta}\log t\}) \\ \hline	
1 & 3 & (-5,8,-2) & (-4.80194, 7.86690, -1.97869) \\
\hline	
2& 10 & (-3,4,0) & (-3.02177, 4.01463, -0.00234) \\
\hline	
3& 31 & (-41,68,-18) & (-40.99761, 67.99839, -17.99974) \\
\hline
4& 100 & (186,-308,81) & (186.00012, -308.00008, 81.00001) \\
\hline	
5& 316 & (-20,74,-63) & (-20.00001, 74.00001, -63.00000) \\
\hline	
6& 1000 & (424,-609,61) & (424.00000, -609.00000, 61.00000) \\
\hline
\end{array}
\end{equation*}

Consistent with our observations above, this data indicates that the error in
approximating $\bm{n}(u_1(t))$ by  $tg_3(\bm{\theta}\log t,\{\bm{\beta}\log t\})$
is roughly on the order of magnitude of $1/t^2$. 
In conclusion, this is an example in which the frequencies with which the elements of $\S$ appear in Theorem A'
are determined quasiperiodically by a linear flow on a two dimensional torus with flow direction determined by $\bm{\theta}$ and a linear flow on $[0,1]^2$ with flow direction determined in $\bm {\beta}$.

\vspace{.15in}

{\footnotesize
\noindent
AH: Department of Mathematics,\\
University of Houston,\\
Houston, TX, United States.\\
haynes@math.uh.edu\\

\noindent
RR: Department of Mathematical Sciences\\
Indiana University--Purdue University Indianapolis,\\
Indianapolis, IN, United States.\\
rroeder@math.iupui.edu
}

\end{document}